\documentclass[12pt]{amsart}
\usepackage{color}

\usepackage{amssymb}
\usepackage{graphicx}  

\newcommand{\CC}{{\mathbb C}}

\newcommand{\ZZ}{{\mathbb Z}}

\newcommand{\RR}{{\mathbb R}}

\newcommand{\defeq}{\stackrel{\rm{def}}{=}}

\newcommand{\NN}{{\mathbb N}}

\newcommand{\rest}{\!\!\restriction}

\renewcommand{\Re}{\mathop{\rm Re}\nolimits}
\renewcommand{\Im}{\mathop{\rm Im}\nolimits}

\theoremstyle{plain}

\newtheorem{thm}{Theorem}
\newtheorem{prop}{Proposition}[section]

\newtheorem{lem}[prop]{Lemma}

\theoremstyle{definition}

\numberwithin{equation}{section}

\def\squarebox#1{\hbox to #1{\hfill\vbox to #1{\vfill}}} 
\newcommand{\stopthm}{\hfill\hfill\vbox{\hrule\hbox{\vrule\squarebox 
                 {.667em}\vrule}\hrule}\smallskip} 
\newcommand{\sech}{\textnormal{sech}}

\newcommand{\indentalign}{\hspace{0.3in}&\hspace{-0.3in}}
\newcommand{\la}{\langle}
\newcommand{\ra}{\rangle}

\usepackage{amsxtra}

\ifx\pdfoutput\undefined
  \DeclareGraphicsExtensions{.pstex, .eps}
\else
  \ifx\pdfoutput\relax
    \DeclareGraphicsExtensions{.pstex, .eps}
  \else
    \ifnum\pdfoutput>0
      \DeclareGraphicsExtensions{.pdf}
    \else
      \DeclareGraphicsExtensions{.pstex, .eps}
    \fi
  \fi
\fi

\title
[Soliton interaction with slowly varying potentials]
{Soliton interaction with slowly varying potentials}

\author[J. Holmer]
{Justin Holmer}
\email{holmer@math.berkeley.edu}
\author[M. Zworski]
{Maciej Zworski}
\email{zworski@math.berkeley.edu}
\address{Mathematics Department, University of California \\
Evans Hall, Berkeley, CA 94720, USA}

\begin{document}

\begin{abstract}
We study the Gross-Pitaevskii equation with a slowly varying
smooth potential, $ V ( x ) = W ( h x ) $. 
We show that up to time $  \log(1/h)/h $ and errors of
size $ h^2 $ in $ H^1 $, the 
solution is a soliton evolving 
according to the classical dynamics of a natural 
effective Hamiltonian, $ (\xi^2 + \sech^2 * V ( x ) )/2 $. This
provides an improvement ($ h \rightarrow h^2 $) compared to 
previous works, and is strikingly confirmed by numerical simulations 
-- see Fig.\ref{f:drama}. 
\end{abstract}

\maketitle

\vspace{-0.35in}
\section{Introduction}   
\label{in}

The Gross-Pitaevskii equation is the nonlinear Schr\"odinger
equation with an external potential:
\begin{equation}
\label{eq:nls}
\left\{
\begin{aligned}
&i\partial_t u + \tfrac{1}{2}\partial_x^2 u - V ( x ) u +u|u|^2 = 0\\
&u(x,0) = e^{ i v_0 x } \sech ( x - a_0 ) \,,  
\end{aligned}
\right.
\end{equation}
In \cite{HZ1} we investigated the case
of $ V ( x )  = h^2 W ( x ) $ where $ 0 < h \ll 1 $ and 
$ W \in H^{-1} ( \RR; \RR) $.
Because of our earlier work on high velocity scattering 
\cite{HMZ1},\cite{HMZ2}, which is potential specific, 
the paper presented the case of $ W ( x ) = 
\delta_0 ( x ) $, but as was pointed out there the method applies verbatim 
to the more general case, $W \in H^{-1}$. 
Motivated by the approach of \cite{HZ1} 
we now revisit the case of slowly varying potentials, $ V = 
W ( h x ) $, and show that using the effective Hamiltonian 
approach we can describe the evolution of the soliton 
with errors of size $ h^2 $.
In particular, in this setting, we improve
the results of Fr\"ohlich-Gustafson-Jonsson-Sigal 
\cite{FrSi}. 

\begin{thm}
\label{t:1}
Suppose that in \eqref{eq:nls},
$ V ( x ) = W ( hx) $, where $ W \in {\mathcal C}^3 ( \RR; \RR) $.
Let $\delta \in ( 0 , 1/2 ) $.
Then on the time interval 
$$0\leq t \leq  
\frac{ \delta  \log ( 1/ h ) } { C h } \,, $$
we have 
\begin{equation}
\label{eq:t1}
\| u ( t ,\bullet ) -  e^{i\bullet v(t) }e^{i\gamma ( t ) } \sech
(\bullet -a ( t) ) \|_{H^1 ( \RR ) } \leq C h^{2- \delta} \,, \ \ 
0 < h < h_0 \,,
\end{equation}
 where 
$ a$, $ v$, $ \gamma$,
solve the following system
of equations
\begin{gather}
\label{eq:t3}
\begin{gathered}
\dot {  a} =   v \,, \ \ 
\dot{  v} = - \frac 12 \sech^2 * V'  (  a) 
\,, 
\\ 
\dot {  \gamma} = \frac 1 2  + \frac {v^2} 2 
-  \sech^2 * V (   a ) +  ( x \, \sech^2 x \tanh x ) * V  (  a) \,, 
\end{gathered}
\end{gather}
with initial data $(a_0,v_0,0)$. 
The constants $C$ and $h_0$ depend  on 
$\|W^{(k)}\|_{L^\infty}$, $0\leq k\leq 3$, and $|v_0|$ only.
\end{thm}

\renewcommand\thefootnote{\dag}%

\begin{figure}
\begin{center}
\includegraphics[width=6in]{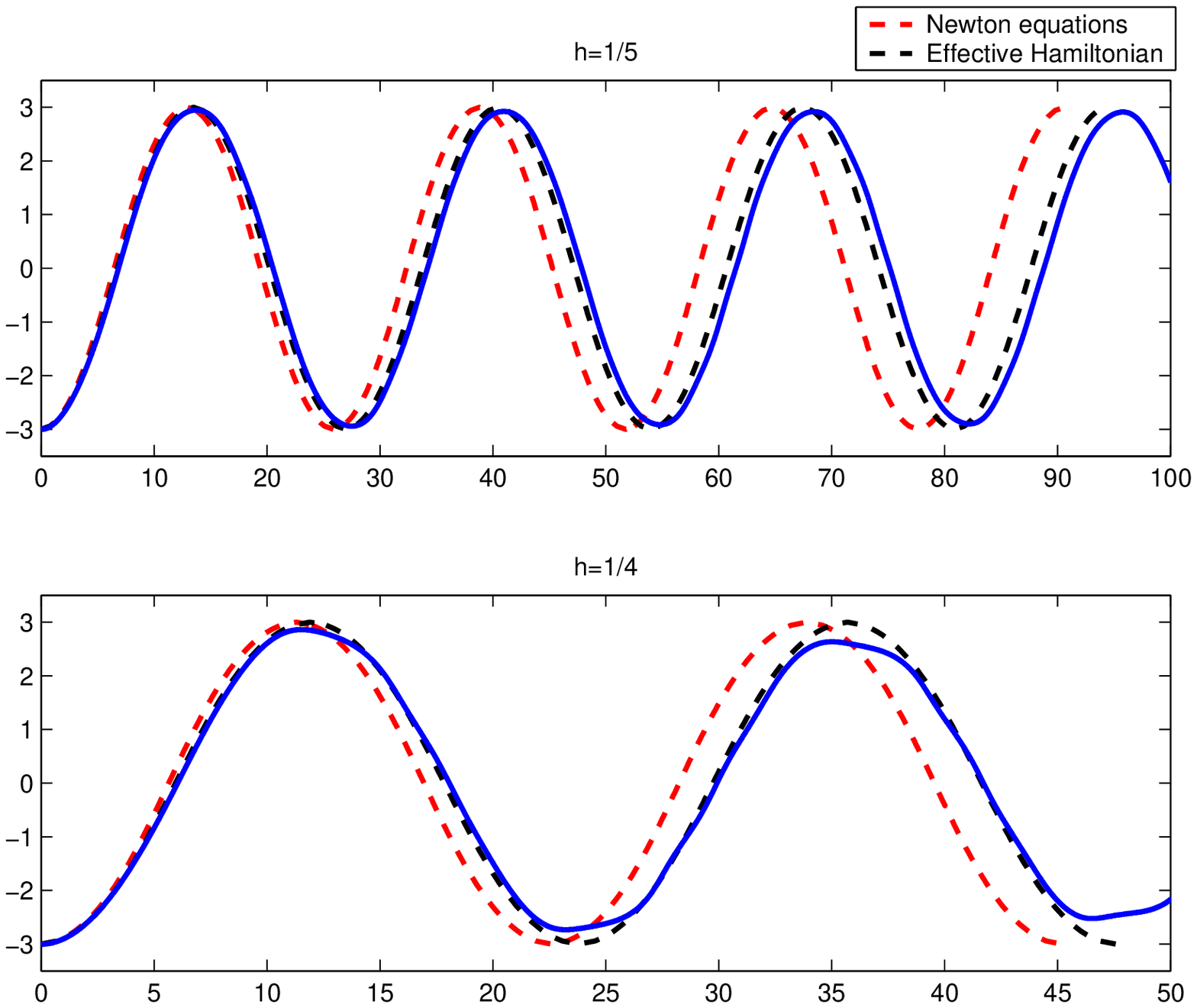}
\end{center}
\caption{Comparison of the dynamics of the center of motion of
the soliton for the Gross-Pitaevskii equation with a slowly 
varying potential, 
$$ iu_t = -\frac12 u_{xx} - |u|^2 u - \sech^2 ( h x ) u \,, \ \ 
h = 1/5 \,, \ \ h=1/4 \,, $$
and initial condition in \eqref{eq:nls} with $ v_0=0$, $ a_0=-3$.
The dashed {{red}} curve shows the solution to Newton's equations
used in \cite{BJ} and \cite{FrSi}, the {{blue}} curve shows
the center of the approximate soliton $ u $,  and the black dashed
curve is given by the equations of motion of the effective 
Hamiltonian
\[  \hspace{-1in} 
\frac 12 \left( v^2  +  \sech^2 ( h \bullet ) * \sech^2 (a )  \right) \,.\]
The improvement of the approximation given by the effective Hamiltonian
is remarkable even in the case of $ h = 1/4 $ in which we already see
radiative dissipation in the first cycle.}
\label{f:drama}
\end{figure}

As in \cite{HZ1} 
the proof of our theorem follows the long tradition of the study 
of stability of solitons which started with the work of M.I. Weinstein
\cite{We}. The interaction of solitons with external potentials was 
studied in the stationary semiclassical setting by 
Floer and A. Weinstein \cite{FlWe} and Oh \cite{Oh}, and the first
dynamical result belongs to Bronski and Jerrard \cite{BJ}, see also
\cite{CaMi},\cite{Ke}.
The semiclassical 
regime is equivalent to considering slowly varying potentials,
and the dynamics in that 
case was studied in \cite{FrSi},\cite{FrSi1},\cite{FrY},\cite{Wal} (see
also numerous references given there).

The results of \cite{FrSi} in the special case of \eqref{eq:nls}
give
\begin{equation}
\label{eq:tslow}
\| u ( t ,\bullet ) -  e^{i\bullet v(t) }e^{i\gamma ( t ) }
\sech(\bullet -a ( t) ) \|_{H^1 ( \RR ) } \leq C h^{1-\delta}  \,,  \ \
0 \leq t \leq \delta \log(1/h)/h \,,
\end{equation}
where
\begin{gather}
\label{eq:t3slow}
\begin{gathered}
\dot {  a} =   v + {\mathcal O} ( h^2 ) \,, \ \
\dot {  v} =  - V' ( a) + {\mathcal O} ( h^2 ) \,,
\\
\dot {  \gamma} = 1/2 + {v^2}/2 
- V(  a ) +  {\mathcal O} ( h^2 ) \,,
\end{gathered}
\end{gather}
with initial data $(a_0,v_0,0)$.  In other words, the motion of
the soliton is approximately given by Newton's equations,
$ \dot a = v $, $ \dot v = - V' ( a ) $. However it is not
clear if, as in \eqref{eq:t3}, the errors 
$ {\mathcal O} ( h^2 ) $ in 
\eqref{eq:t3slow} can be removed without affecting \eqref{eq:tslow}
-- see \S \ref{ode}. Since \eqref{eq:t3} imply equations \eqref{eq:t3slow},
Theorem \ref{t:1} shows that we can replace $ h^{1-\delta } $ by 
$ h^{2-\delta} $ in \eqref{eq:tslow}, keeping \eqref{eq:t3slow}.

When $ V ( x) = h^2 W ( x) $,
and $ W \in H^{-1} ( \RR ; \RR) $ \cite{HZ1}, 
Newton's equations are clearly not applicable. To 
describe evolution up to time $ \log ( 1/h ) /h $ we introduced a
natural effective Hamiltonian. A numerical experiment shown 
in \cite[Fig.2]{HZ1} and reproduced here in Fig.\ref{f:drama} 
suggested that the effective Hamiltonian approach gives a dramatic
improvement for slowly varying potentials. Theorem \ref{t:1} 
(and a more precise Theorem \ref{t:2} below) quantify that
improvement by changing the errors from $ h $ in \eqref{eq:tslow} to 
$ h^2 $ in \eqref{eq:t1}. 

To describe the natural effective Hamiltonian 
we recall  that the 
Gross-Pitaevski equation \eqref{eq:nls} is the 
equation for the Hamiltonian flow of 
\begin{equation}
\label{eq:HV}
 H_V ( u ) \defeq \frac14 \int ( |\partial_x u |^2 - |u|^4 ) dx
+ \frac12 \int V  | u|^2 \,, \ \ 
\end{equation}
with respect to the symplectic form on $ H^1 ( \RR , \CC ) $ (considered
as a real Hilbert space):
\begin{equation}
\label{eq:ome1}
 \omega ( u , v ) = \Im \int u \bar v \,, \ \ u , v \in H^1 ( \RR , \CC) 
\,.\end{equation}
When $ V \equiv 0 $, $ \eta = \sech $ is a 
minimizer  
of $ H_0 $ with the prescribed $ L^2 $ norm ($\| \eta \|^2_{L^2} = 2 $):
\begin{equation}
\label{eq:defE} d {\mathcal E}_{\eta } = 0 \,, \ \ {\mathcal E}( u ) 
\defeq H_0 ( u ) + \frac14 \| u \|_{L^2}^2 \,.
\end{equation}
Here $ d {\mathcal E}_\eta $ is the differential of $ {\mathcal E}
\;:\; H^1 \rightarrow \RR $, see \S \ref{s:sy}.

The flow of $ H_0 $ is tangent to the 
manifold of solitons,
\[  M = \{ e^{i \gamma} e^{ i v ( x - a ) } \mu \, 
\sech ( \mu ( x - a ) ) \,, \ \ 
a, v, \gamma \in \RR \,, \ \ \mu \in \RR_+ \} \,, \]
which of course corresponds to the fact that the solution of \eqref{eq:nls}
with $ V = 0 $ and $ u_0 ( x , 0 ) = e^{ i \gamma_0 + i v_0 ( x - a_0 ) } \mu
 \sech ( \mu ( x - a_0 ) ) $, is
\begin{gather}
\label{eq:frees} 
\begin{gathered}
 u ( x , t) = g(t)\cdot \eta \defeq
e^{ i \gamma + i v_0  ( x - a_0) + i ( \mu^2 - v^2) t/ 2 } 
\mu \, \sech ( \mu ( x - a_0 - v_0 t ) 
)
\,, 
\\
g ( t) \defeq ( a_0 + v_0 t, v_0, \gamma_0 + ( \mu^2 +v^2 )/2 , \mu ) \,.
\end{gathered}
\end{gather}
The symplectic form \eqref{eq:ome1} restricted to $ M $ is
\begin{equation}
\label{eq:ome2} \omega\rest_M
 = \mu dv\wedge da + vd\mu\wedge da + d\gamma \wedge d\mu
\,,\end{equation}
see \S \ref{ssms}. The evolution of the parameters $ ( a, v ,\gamma, \mu ) $
in the solution \eqref{eq:frees} follows the Hamilton flow of 
\[  H_0\rest_M = \frac { \mu v^2 }2 - \frac{\mu^3 }6 \,, \]
with respect to the symplectic form $ \omega\rest_M $. 

 The system of equations \eqref{eq:t3} is
obtained using the following basic idea: if a Hamilton flow of $ H $, 
with initial condition on a symplectic submanifold, $ M $, 
stays close to $ M $, then the flow is close to the Hamilton 
flow of $ H \rest_M $.  
In our case $ M$ is the manifold of solitons and $ H $ is given 
by \eqref{eq:HV}
\begin{equation}
\label{eq:Heff}
 H_V\rest_M ( a ,v, \gamma, \mu ) =
\frac { \mu v^2 }2 - \frac{\mu^3 }6 +
\frac12 \mu^2 (V  * \sech^2) ( \mu a ) \,.
\end{equation}
The equations \eqref{eq:t3} are simply the equations of the flow of 
$ H_{V } \rest_M $ -- see \S \ref{s:GPH}. 
They are easily seen to imply
\eqref{eq:t3slow} 
but some $ h$ corrections are built into the 
classical motion leading to the improvement in Theorem \ref{t:1}, 
see also Fig.\ref{f:drama}.

As in previous works all of this hinges on the proximity of 
$ u ( x , t) $ to the manifold of solitons, $ M $. In 
\cite{HZ1} we followed \cite{FrSi} and used 
Weinstein's Lyapunov function \cite{We},
\[ L ( w ) \defeq {\mathcal E} ( w + \eta ) - {\mathcal E} ( \eta ) \,, \]
where $ {\mathcal E} $ is given by \eqref{eq:defE}.
Here, in the notation similar to \eqref{eq:frees}, 
\[ u ( t) = g(t ) \cdot ( \eta + w ( t) ) \,, \]
for an optimally chosen $ g ( t) = ( a ( t) , v( t ) , \gamma ( t) , 
\mu ( t ) ) $ -- see Lemma \ref{l:stan}.

The use of $ L ( w ) $ seems essential for the all-time orbital stability
of solitons. Up to times $ h^{-1} \log ( 1/h) $
we found that it is easier to use its quadratic approxition
\begin{equation}
\label{eq:defL}
 L_0 \defeq \langle {\mathcal L} w , w \rangle \,, \ \  {\mathcal L} w 
\defeq {\mathcal E}''_\eta ( w )  = 
 - \frac12 \partial_x^2 w - 2 \eta^2 w - \eta^2 \bar w 
+ \frac12 w \,. 
\end{equation}
Rather than use conservation of energy, $ H_V ( u ) $, we use the
nonlinear equation for $ w ( t )$ in estimating $ L_0 ( w) $ -- 
see \S \ref{pr}. That involves solving a nonhomogeneous linear equation
approximately using the spectral properties of $ {\mathcal L }$
-- see \eqref{E:forcedlinear} and Proposition \ref{p:new}. 
The fact that the solution of that equation is of size $ h^2 $ in 
$ H^1 $ gives the first indication of the improvement based on 
using the effective Hamiltonian.
Ultimately, this makes the argument simpler than 
in the case of $ W \in H^{-1} $ (or the special case of $ W = \delta_0 $).

The paper is organized as follows. In \S \ref{rhs} we recall the Hamiltonian
structure of the nonlinear flow of \eqref{eq:nls} and describe the
manifold of solitons. As in \cite{HZ1}, 
its identification with the Lie group $ G = 
H_3 \ltimes \RR_+ $, where $ H_3 $ is the Heisenberg group, provides
useful notational shortcuts. In \S \ref{re} we describe
the reparametrized evolution. The starting point there is an 
application of the implicit function theorem and a decomposition of the 
solution into symplectically orthogonal components. That method
has a long tradition in soliton stability and we learned it from 
\cite{FrSi}. The analysis of the orthogonal component using 
an approximate solution to a linear equation and a bootstrap
argument are presented in \S \ref{pr}. This results in a 
somewhat more precise version of Theorem \ref{t:1} -- see Theorem \ref{t:2}.
The ODE estimates needed for the exact evolution \eqref{eq:t3} 
are given in \S \ref{ode}. Finally we show how Theorem \ref{t:2} implies
Theorem \ref{t:1}. Except for basic material such 
as properties of Sobolev spaces or elementary symplectic geometry, 
and a reference
to the proof of Proposition \ref{p:coer}, 
the paper is meant to be self contained. 

\noindent
{\sc Acknowledgments.} 
The work of the first author was supported in part by an NSF postdoctoral 
fellowship, and that
of the second second author by the NSF grant DMS-0654436.


\section{The Hamiltonian structure and the manifold of solitons}
\label{rhs}

In this section we recall the well known facts about the
Hamiltonian structure of the nonlinear Schr\"odinger equation.
The manifold of solitons is given as an orbit of a semidirect
product of the Heisenberg group and $ \RR_+ $.

\subsection{Symplectic structure}
\label{s:sy}

In our work, we consider
\[ V \defeq H^1(\mathbb{R}, \mathbb{C})\subset L^2(\mathbb{R},\mathbb{C})\,,\]
viewed as a {\em real} Hilbert space.
The inner product and the symplectic form are given by 
\begin{equation}
\label{eq:omega}\
\langle u,v \rangle  \defeq \Re \int u\bar v \,, \  \ 
\omega(u,v) \defeq \langle i u , v \rangle = \Im \int u\bar v\,, 
\end{equation}
Let 
$H:V\to \mathbb{R}$ be a function, a Hamiltonian.  
The associated Hamiltonian
vector field is a map $\Xi_H : V\to TV$, 
which means that for a particular point $u\in V$, we have 
$(\Xi_H)_u \in T_uV$. The vector field  $\Xi_H$ is defined by the relation
\begin{equation}
\label{eq:Hamvf} \omega(v , (\Xi_H)_u) = d_uH(v)
\,, \end{equation}
where $v\in T_uV$, and $d_uH:T_uV \to \mathbb{R}$ is defined by
$$d_uH(v) = \frac{d}{ds}\Big|_{s=0} H(u+sv) \,. $$
In the notation above
\begin{equation}
\label{eq:nats}  dH_u ( v ) = \langle dH_u , v \rangle \,, 
\ \  (\Xi_H)_u = \frac 1 i dH_u  \,. 
\end{equation}

If we take $ V = H^1 ( \RR , \CC ) $ with the symplectic form \eqref{eq:omega},
and
$$H(u) = \int \frac14|\partial_x u|^2 - \frac14|u|^4$$
then we can compute
\begin{align*}
d_uH(v) &= \Re \int ( (1/2) \partial_x u \partial_x \bar v - |u|^2u \bar v)\\
&= \Re \int ( - (1/2)\partial_x^2 u - |u|^2u)\bar v \,. 
\end{align*}
Thus, in view of \eqref{eq:nats} and \eqref{eq:Hamvf}, 
$$(\Xi_H)_u = \frac 1 i \left( - \frac12 \partial_x^2u - |u|^2u \right) $$
The flow associated to this vector field (Hamiltonian flow) is
\begin{equation}
\label{eq:Hflow}
\dot u = (\Xi_H)_u =  \frac 1 i  \left( - 
\frac12 \partial_x^2u - |u|^2u \right) 
\,.
\end{equation}

For future reference we state two general lemmas of symplectic 
geometry. The simple proofs can be found in \cite[\S 2]{HZ1}.
\begin{lem}
\label{l:gen1}
Suppose that $ g : V \rightarrow V $ is a diffeomorphism such 
that $ g^* \omega = \mu ( g ) \omega $, where $ \mu( g ) \in C^\infty ( 
V ; \RR)  $. Then for $ f \in C^\infty ( V , \RR ) $,
\begin{equation} 
\label{eq:lg1}
(g^{-1})_* \Xi_f(g(\rho)) = \frac{1}{\mu(g)}\Xi_{g^*f}(\rho) \,, \ \ 
\rho \in V \,. 
\end{equation}
\end{lem}

Suppose that $ f \in C^\infty ( V ; \RR ) $ and that $ df ( \rho_0  ) = 0 $.
Then the Hessian of $ f $ at $ \rho_0 $, 
$ f'' ( \rho_0 ) : T_\rho V \mapsto T^*_\rho V $, is well defined. 
We can identify $ T_\rho V $ with $ T^*_\rho V $ using the innner product, 
and define
the Hamiltonian map $ F : T_\rho V \rightarrow T_\rho V $ by
\begin{equation} 
\label{eq:HamF}
 F = \frac 1 i  f'' ( \rho_0 ) \,, \ \  \langle f''(\rho_0 ) X , Y \rangle
  = \omega ( Y , F X )
\,. 
\end{equation}
In this notation we have 
\begin{lem}
\label{l:gen2}
Suppose that $ N \subset V $ is a finite dimensional symplectic submanifold of 
$ V $, 
and $ f \in C^\infty ( V , \RR )$ satisfies 
\[  \Xi_f ( \rho ) \in T_\rho N \subset T_\rho V \,, \ \ \rho \in N \,.\]
If at $ \rho_0 \in N $, $ df ( \rho_0) = 0 $, then the Hamiltonian map
defined by \eqref{eq:HamF} satisfies
\[  F ( T_\rho N ) \subset T_\rho N  \,. \]
\end{lem}

The Hamiltonian map, $ F $, is simply the linearization of $ \Xi_f $ 
at a critical point of $f$. 
An example relevant to this paper is 
\[ f ( u ) = {\mathcal E} ( u ) \defeq H_0 ( u ) + \frac 14 \| u \|^2 \,, \]
see \eqref{eq:defE}. The soliton $ \eta $ is a critical point of
$ {\mathcal E} $ and the Hessian of $ E $ is given by $ {\mathcal L} $ in 
\eqref{eq:defL}. The Hamiltonian map $ F = (1/i) {\mathcal L } $ is
the linearization of $ \Xi_f $ at $ \eta $. In other words, 
$ ( 1/ i ) (  {\mathcal L} - 1/2 ) $ is the linearization of 
\eqref{eq:nls} (with $ V = 0 $) at $ \eta$. The $ 1/2 $ term comes from 
$ \| u \|^2 /4 $ in the definition of $ {\mathcal E}$. 

In Lemma \ref{l:gen2} we can take $ N $ to be the four dimensional
manifold of solitons and $ \rho = \eta $. It then says that 
$ ( 1/i ) {\mathcal L} $ preserves the symplectic orthogonality
of $ w \in T_\eta V $ to $ T_\eta M $.

\subsection{Manifold of solitons as an orbit of a group}

For $ g = ( a, v, \gamma , \mu ) \in \RR^3 \times \RR_+ $ we define
the following map 
\begin{equation}
\label{eq:repG}    H^1 \ni u \longmapsto g\cdot u \in H^1 \,, \ \ 
(g\cdot u)(x) \defeq e^{i\gamma}e^{iv(x-a)}\mu u(\mu(x-a)) \,. 
\end{equation}
This action gives a group structure on $ \RR^3 \times \RR_+ $
and it is easy to check that this
transformation group is a semidirect product of the Heisenberg group
$ H_3 $ and $ \RR_+ $:
\[ G= H_3\ltimes\mathbb{R}_+ \,, \ \ 
\mu\cdot(a,v,\gamma) = (\frac{a}{\mu} ,\mu v, \gamma) \,.\]
Explicitly, the group law on $ G $ is given by 
\[ (a,v,\gamma,\mu) \cdot (a',v',\gamma',\mu') = (a'',v'',\gamma'',\mu'') \,, 
\] 
where 
\[ \begin{split}
& v'' = v + v'\mu \,, \ \ 
 a'' = a + \frac{a'}{\mu} \,, \ \
 \gamma'' = \gamma + \gamma' + \frac{va'}{\mu} \,, \ \
 \mu'' = \mu \mu'
\end{split} \]

The action of $ G $ is not symplectic but it is {\em conformally symplectic}
in the sense that 
\begin{equation} 
\label{eq:conf} g^*\omega = \mu(g)\omega
\,, \ \ g = ( h ( g ) , \mu ( g ) ) \,, \ \ \mu ( g ) \in \RR_+ \,,
\end{equation}
as is easily seen from \eqref{eq:omega}.

The Lie algebra of $ G$, denoted by $ {\mathfrak g } $, 
is generated by $ e_1, e_2, e_3 , e_4 $, 
\[ \begin{split} & \exp( t e_1 ) = (t ,0,0,1)  \,, \ \
\exp( t e_2 ) = (0 ,t,0,1)  \,,
 \\ &  \exp(t e_3 ) = ( 0, 0 , t, 1 ) \,, \ \ \exp ( t e_4 )
= ( 0 , 0 , 0 , e^t) \,, \end{split} \]
and the bracket acts as follows:
\begin{equation}
\label{eq:liea}
[e_1,e_4]=e_1, \quad [e_2,e_4] = -e_2, \quad [e_1,e_2]=-e_3, 
\quad [e_3,\bullet ] =0 \,, \end{equation}
 so $e_3$ is in the center. 
The infinitesimal representation obtained from \eqref{eq:repG} 
is given by
\begin{equation}
\label{eq:liea1}  e_1 = -\partial_x \,, \ \
e_2 = ix \,,  \ \ e_3 = i \,, \ \  e_4 = \partial_x \cdot x \,. 
\end{equation}
It acts, for instance on $ {\mathcal S}( \RR ) \subset H^1 $, 
and by $ X \in {\mathfrak g} $ we will denote a linear combination of
the operators $ e_j $.

The proof of the following standard lemma can be again found in 
\cite[\S 2]{HZ1}:
\begin{lem}
\label{l:stan}
Suppose $ \RR \ni t \mapsto g ( t) $ is a $ C^1 $ function and that
$ u \in {\mathcal S} ( \RR ) $. Then, in the notation of \eqref{eq:repG},
\[  \frac{d}{dt } g(t) \cdot u = g( t ) \cdot ( X ( t ) u ) \,, \]
where $ X( t ) \in {\mathfrak g } $ is given by 
\begin{equation}
\label{eq:lstan}
X ( t ) =   \dot a  ( t) \mu ( t) e_1 +
 \frac {\dot v (t) } { \mu ( t ) } e_2 + 
(\dot \gamma ( t) - \dot a ( t ) v ( t )  ) e_3 + 
\frac  {\dot\mu( t) }{ \mu ( t ) } e_4 \,, 
\end{equation}
where $ g ( t) = ( a ( t ) , v ( t) , \gamma( t ) , \mu ( t) ) $.
\end{lem}

The manifold of solitons is an orbit of this group, $ G \cdot \eta $, 
to which 
$ \Xi_H $, defined in \eqref{eq:Hamvf}, is tangent. In view of
\eqref{eq:Hflow} that means that 
\[  i \left( \frac12 \partial_x^2 \eta + |\eta|^2 \eta \right) = 
X \cdot \eta \,, \]
for some $ X \in {\mathfrak g } $. The simplest choice is given 
by taking $ X = \lambda i  $, $ 
\lambda \in \RR $, so that $ \eta $ solves
a nonlinear elliptic equation 
$$-\frac12 \eta'' - \eta^3 + {{\lambda}}\eta =0 \,. $$
This has a solution in $ H^1 $ if $ \lambda = \mu^2/2 > 0 $ and it then is
$\eta(x)= \mu \sech(\mu x)$.
We will fix $ \mu = 1$ so that 
$$\eta(x)=  \sech x \,. $$ 
Using Lemma \ref{l:gen1} we can check that $ G\cdot \eta $ is
the {\em only} orbit of $ G $ to which $ \Xi_H $ is tangent. 

We define the submanifold of solitons, $M \subset H_1$,
as the orbit of $\eta$ under $G$, 
$$M = G\cdot \eta \subset H_1$$
and thus we have the identifications
\begin{equation}
\label{eq:idG}
  M = G\cdot \eta \simeq G / \ZZ \,, \ \  T_\eta M = {\mathfrak g} \cdot 
\eta \simeq  {\mathfrak g }  \,. 
\end{equation}
The quotient corresponds to the $\ZZ$-action 
$$ ( a, v, \gamma, \mu)  \mapsto (a, v, \gamma + 2 \pi k , \mu ) \,, \ \ 
 k \in \ZZ $$

\subsection{Symplectic structure on the manifold of solitons}
\label{ssms}
We first compute the symplectic form $ \omega \rest_M $ on 
$ T_\eta M $ using the identification \eqref{eq:idG}:
\[ (\omega\rest_M )_\eta ( e_i , e_j ) = \Im \int (e_i \cdot \eta )(x)
(\overline {e_j \cdot \eta} ) ( x ) \,.\]
Since 
\[ \int \eta^2 ( x ) dx = 2 \,, \ \ \int \eta ( x) \partial_x \eta ( x ) = 0 
\,, \ \  \int \partial_x \eta ( x ) x \eta ( x ) dx = -1 \,, \]
we obtain from \eqref{eq:liea1} that 
\begin{equation}
\label{eq:omij}
  (\omega\rest_M )_\eta ( e_2 , e_1 ) = 1 \,, \ \ (\omega\rest_M)_\eta 
( e_3, e_4 ) = 1 \,, 
\end{equation}
and all the other $   (\omega\rest_M )_\eta ( e_i , e_j ) $'s vanish.
In other words,
\[    (\omega\rest_M )_\eta = (dv \wedge da + d \gamma \wedge d \mu )_
{ ( 0 ,0 , 0 , 1 ) } = ( d ( v da + \gamma d\mu ) )_{ ( 0, 0 , 0 , 1 )} \,. \]
Using \eqref{eq:idG} we conclude that 
\begin{equation}
\label{E:Momega}
\omega \rest_M  = \mu dv\wedge da + vd\mu\wedge da + d\gamma \wedge d\mu \,,
\end{equation}
see \cite[Lemma 2.3]{HZ1}.

Now let $f$ be a function defined on 
$M$, $ f = f ( a, v, \gamma, \mu ) $.
The associated Hamiltonian vectorfield, $\Xi_f$, is defined by 
$$\omega(\cdot, \Xi_f) = 
df = f_ada+f_vdv+f_\mu d\mu + f_\gamma d\gamma \,. $$
Using \eqref{E:Momega} we obtain 
\begin{equation}
\label{eq:XifM}
\Xi_f = \frac{f_v}{\mu} \partial_a 
+ \left( -\frac{f_a}{\mu}-\frac{vf_\gamma}{\mu}\right)\partial_v 
+ f_\gamma \partial_\mu + \left(v\frac{f_v}{\mu}-f_\mu\right)\partial_\gamma
\,. \end{equation} 
The Hamilton flow is obtained by solving
\[ 
\dot v = -\frac{f_a}{\mu}-\frac{vf_\gamma}{\mu}\,, \ \
\dot a = \frac{f_v}{\mu} \,, \ \
\dot \mu = f_\gamma \,, \ \ 
\dot \gamma = v \frac{f_v} \mu - f_\mu \,. \]
The restriction of 
$$ H ( u) = \frac14\int |\partial_x u|^2 - \frac14\int |u|^4$$ 
to $ M $ is given by computing by 
\begin{equation}
\label{eq:ffM}
 f ( a, v, \gamma, \mu ) = H ( g\cdot \eta ) = 
 \frac{\mu v^2}{2}-\frac{\mu^3}{6} \,, \ \ g = ( a, v , \gamma, \mu ) \,. 
\end{equation}
The flow of \eqref{eq:XifM} for this $ f $ describes the evolution of
a soliton.

\subsection{The Gross-Pitaevski Hamiltonian restricted to the 
manifold of solitons}
\label{s:GPH}

We now consider the Gross-Pitaevski Hamiltonian,
\begin{equation}
\label{eq:GPH}
H_q ( u ) \defeq \frac14 \int ( |\partial_x u |^2 - |u|^4 ) dx 
+ \frac12 \int V ( x) | u |^2 dx  \,, 
\end{equation}
and its restriction to $ M = G \cdot \eta $:
\begin{equation}
\label{eq:GPHM} 
{H_q }\rest_M = f ( a , v , \gamma, \mu ) = 
\frac { \mu v^2 }2 - \frac{\mu^3 }6 + 
\frac{\mu^2} 2   V * ( \sech^2 ( \mu \bullet ) ) ( a  ) \,. 
\end{equation}
The flow of $ (H_q)\rest_M $ can be read off from \eqref{eq:XifM}:
\begin{gather}
\label{eq:GPfl}
 \begin{gathered}
\dot v = -\frac{f_a}{\mu} - \frac{v f_\gamma}{\mu} = - \frac{\mu^2} 2 
V' * ( \sech^2 ( \mu \bullet ) ) ( a  )  \,, \ \ 
\dot a  = \frac{f_v}{\mu} = v \,, \ \ \  \dot \mu = f_\gamma = 0 \,, \\
 \dot \gamma  = v \frac{f_v} \mu - f_\mu  
= \frac12 v^2 +  \frac 12 \mu^2 - \mu
  V * ( \sech^2 ( \mu \bullet) ) ( a  )  
+ \mu V* ( x\, \sech^2(x)\tanh(x) \rest_{x=\mu \bullet }) ( a ) \,. 
\end{gathered} \end{gather}
This are the same equations as \eqref{eq:t3}. The evolution of $ a $ and 
$ v $ is simply the Hamiltonian evolution of 
$ ( v^2 + \mu^2 V * \sech^2 ( \mu \bullet) ( a))/2  $, $ \mu = \text{const} $. 
The more mysterious evolution of the phase $ \gamma $ is now explained
by \eqref{eq:GPHM}.

We can also rewrite \eqref{eq:GPfl} as follows:
$$f = \frac{\mu v^2}{2} - \frac{\mu^3}{6} 
+ \frac{\mu}{2}\int V\Big( \frac{x}{\mu} + a\Big) \eta^2(x)dx\,, $$ 
so that 
\begin{align*}
\dot \gamma &= \frac{v^2}{2}+\frac{\mu^2}{2} 
- \frac12\int V\Big(\frac{x}{\mu}+a\Big)\eta^2(x)\,dx 
+\frac12\int V'\Big(\frac{x}{\mu}+a\Big)\frac{x}{\mu}\eta^2(x)\, dx\\
&= \frac{v^2}{2}+\frac{\mu^2}{2}-V(a) + \frac{V''(a)\pi^2}{24\mu^2} 
+ \mathcal{O}(h^4) \,
\end{align*}
and
\begin{align*}
\dot v &= -\frac12\int V'\Big( \frac{x}{\mu}+a\Big)\eta^2(x)\,dx 
= -V'(a) - \frac{V'''(a)\pi^2}{24\mu^2} + \mathcal{O}(h^5)\, ,
\end{align*}
where we used that $\int \eta^2 =2$ and $\int x^2\eta^2 = \pi^2/6$ in the 
Taylor expansions. Here for the purpose of presentation we assumed that
$ W \in {\mathcal C}^5 $ but we will never use Taylor's formula with more
than four terms, for which $ {\mathcal C}^3 $ is only needed.

\section{Reparametrized evolution}
\label{re}

To see the effective dynamics described in \S \ref{s:GPH} we 
write the solution of \eqref{eq:nls} as 
\[  u ( t) = g ( t ) \cdot ( \eta + w ( t )  ) \,,  \ \ 
w ( t) \in H^1 ( \RR, \CC ) \,, \]
where $ w ( t) $ satisfies
\[  \omega ( w ( t) , X \eta ) = 0 \,, \ \ \forall X \in {\mathfrak g}\,.\]
To see that this decomposition is possible, 
initially for small times, we apply the 
following consequence of the implicit function theorem and the nondegeneracy 
of $ \omega\rest_M $ (see \cite[Proposition 5.1]{FrSi} for a more
general statement):
\begin{lem} 
\label{l:FGJS}
For $\Sigma \Subset G/\ZZ$ (where the topology on $G/\ZZ $ is given by 
the identification with  $ \mathbb{R}\times 
\mathbb{R} \times S^1 \times \mathbb{R}_+$) let 
$$U_{\Sigma,\delta} = \{ \, u \in H_1 \, : \, 
\inf_{g\in\Sigma}\|u -g\cdot \eta\|_{H^1}<\delta \} \,. $$
If $ \delta \leq \delta_0 = \delta_0 (\Sigma)$ then for any $ 
u \in U_{\Sigma,\delta}$, there exists a unique 
$  g( u )\in \Sigma $ such that
\begin{equation}
\label{E:gdef}
 \omega(g(u)^{-1} \cdot u -  \eta, X \cdot \eta ) = 0 \quad \forall
 X \in {\mathfrak g } \,.
\end{equation}
Moreover, the map $ u  \mapsto g( u )$ is in 
$C^1(U_{\Sigma,\delta},\Sigma)$. 
\end{lem}
\begin{proof}
We define the transformation
\[ F \; : \; H^1 ( \RR , \CC ) \times G \; \longrightarrow \; 
{\mathfrak g}^* \,, \ \ [ F ( u , h ) ]( X)  \defeq 
\omega ( h \cdot u - \eta , X \cdot \eta ) \,,\]
and want to solve 
$ F ( u , h ) = 0 $ for $ h = h ( u ) $. By the implicit
fuction theorem that follows for $ u $  near $ G \cdot \eta $ if for any 
$ g_0 \in G $ the linear transformation
\[  d_h F ( g_0 \cdot \eta , g_0 ) \; : \; T_{g_0} G \longrightarrow {\mathfrak
g}^* \,, \]
is invertible. Clearly we only need to check it for $ g_0 = e$, that is
that 
$  d_h F ( \eta , e ) \; : \; {\mathfrak g}  \rightarrow {\mathfrak
g}^* \,, $ 
is invertible. But as an element of 
$ {\mathfrak g}^* \otimes {\mathfrak g}^* $, 
$ d_h F( \eta , e ) = (\omega\rest_M)_\eta $, which is nondegenerate.
\end{proof}

From \S\S \ref{s:sy} and \ref{s:GPH} we recall that the equation for 
$ u $ \eqref{eq:nls} can be written as 
\begin{equation} 
\label{eq:utH}
\partial_t u  = \Xi_{H_q} (u ) \,, \ \ 
H_q ( u ) \defeq \frac14 \int ( |\partial_x u |^2 - |u|^4 ) dx 
+ \frac12 \int V( x ) | u |^2 dx \,. \end{equation}
Using Lemma \ref{l:FGJS} we define
\begin{equation}
\label{eq:wt}  g ( t ) \defeq g ( u ( t) ) \,, 
 \ \ \tilde u \defeq g(t)^{-1} u (t) \,, \ \ 
w(t) \defeq \tilde u  - \eta \,, 
\end{equation}
and we want to to derive an equation for $ w ( t ) $.

Let
\begin{equation}
\label{eq:albe}
\begin{split}
\alpha = \alpha ( a , \mu ) &\defeq \frac12\int V\Big(\frac{x}{\mu}+a\Big)\eta^2(x)dx 
- \frac12\int V'\Big(\frac{x}{\mu}+a\Big)\frac{x}{\mu}\eta^2(x)\,dx \,, \\
\beta = \beta ( a , \mu ) 
& \defeq \frac{1}{2\mu}\int V'\Big(\frac{x}{\mu}+a\Big)\eta^2(x)\, dx \,. 
\end{split}
\end{equation}
Note that 
dependence on $a$, $\mu$ makes $\alpha$, $\beta$ into
time-dependent parameters. They are however independent of $x$.  
The proper motivation for the choice of $ \alpha $ and $ \beta $ 
will come in Lemma \ref{l:new} below.

Set
\begin{equation}
\label{eq:defX} 
X=(-\dot a+v)\mu e_1 + \Big(-\frac{\dot v}{\mu}-\beta\Big)e_2 
+ \Big(-\dot \gamma +\dot a v -\frac12v^2+\frac12\mu^2-\alpha\Big)e_3
-\frac{\dot \mu}{\mu}e_4 \,, 
\end{equation}
where $ e_j $'s are given by \eqref{eq:liea1}. Let also
\begin{equation}
\label{eq:defN}
\mathcal{N}w = 2|w|^2\eta+\eta w^2+|w|^2w \,. 
\end{equation}
We now have 
\begin{lem}
\label{l:wt}
The equation for $w$, defined by \eqref{eq:wt}, is
$$\partial_t w =
\begin{aligned}[t]
&X\eta +i\Big[-V\Big(\frac{x}{\mu}+a\Big)+\alpha +\beta x\Big]\eta \\
&+Xw+i\Big[-V\Big(\frac{x}{\mu}+a\Big)+\alpha +\beta x\Big]w 
-i\mu^2\mathcal{L}w +i\mu^2\mathcal{N}w\,, 
\end{aligned} $$
where $ X $ is given by \eqref{eq:defX}, $ \mathcal L $ by \eqref{eq:defL},
$ {\mathcal N} $ by \eqref{eq:defN}.
\end{lem}
\begin{proof}
 We compute, by the chain rule
$$\partial_t u = \partial_t (g\cdot(\eta+w)) = g\cdot Y(\eta+w) 
+ g\cdot \partial_t w$$
where 
$$Y = \dot a\mu e_1 + \frac{\dot v}{\mu}e_2 + (\dot \gamma-\dot av)e_3 
+ \frac{\dot \mu}{\mu}e_4 \,.$$
From this, and the fact that the equation for $u$ can be written 
$\partial_t u = \Xi_{H}(u)$, we get
$$\partial_t w = -Y\tilde u + g^{-1}\partial_t u 
= -Y\tilde u +g^{-1}\Xi_H g\tilde u$$
We apply Lemma \ref{l:gen1} to obtain 
$$\partial_t w = -Y\tilde u + \frac{1}{\mu}\Xi_{g^*H} \tilde u$$
We compute
\[ \begin{split}
(g^*H)(\tilde u) = H(g\tilde u) 
= & 
\frac14v^2\mu \|\tilde u\|_{L^2}^2 
+ \frac14\mu^3\|\partial_x\tilde u\|_{L^2}^2 
+\frac12v\mu^2\Im \int \bar{\tilde u}\partial_x \tilde u \, dx \\
\ \ \ \ \  \ \ \ \ \ \ \ & - \frac14\mu^3\|\tilde u\|_{L^4}^4 
+ \frac12\mu \int V\Big(\frac{x}{\mu}+a\Big)|\tilde u(x)|^2 \, dx
\end{split}
\]
Therefore,
$$\Xi_{g^*H}\tilde u = \frac1i\Big( \frac12v^2\mu \tilde u 
- \frac12\mu^3\partial_x^2\tilde u - \mu^3|\tilde u|^2 \tilde u 
+\mu V\Big(\frac{x}\mu + a\Big)\tilde u\Big) - v\mu^2\partial_x \tilde u$$
Substituting and expanding the cubic term,
$$
\partial_t w = 
\begin{aligned}[t]
&-Y(\eta+w) - \frac12iv^2(\eta+w) + \frac12i\mu^2\partial_x^2(\eta+w) 
+i\mu^2[\eta^3+2\eta^2w+\eta^2\bar w \\
&+ 2|w|^2\eta+\eta w^2 + |w|^2w] - iV\Big( \frac{x}{\mu}+a\Big)(\eta+w) 
- v\mu\partial_x(\eta+w)
\end{aligned}
$$
Using that $-\frac12\eta+\frac12\eta''+\eta^3=0$, we obtain
$$
\partial_tw = \begin{aligned}[t]
&-Y\eta - \frac12iv^2\eta+\frac12i\mu^2\eta - iV\Big(\frac{x}{\mu}+a\Big)\eta 
-v\mu\partial_x \eta\\
&-Yw-\frac12iv^2w+\frac12i\mu^2w-i\mu^2\mathcal{L}w+i\mu^2\mathcal{N}w 
-iV\Big(\frac{x}{\mu}+a\Big)w - v\mu\partial_xw
\end{aligned}
$$
Now set
$$X =-Y +v\mu e_1 - \beta e_2 +\Big[-\frac12v^2+\frac12\mu^2-\alpha\Big]e_3$$
From this, we get the claimed equation.
\end{proof}

Let us make some remarks about the lemma. First, note that if $X=0$, 
then 
\[ \dot v = -\beta \mu \,, \ \ \dot \gamma = \dot a v -\frac12 v^2  
+ \frac12\mu^2-\alpha\,, \ \ \dot a =v\, , 
 \ \ b\dot \mu =0 \,, \]
 which are exactly the 
equations of motion of the effective Hamiltonian -- see \S \ref{s:GPH}. 

Second, note that 
the term
$$i\Big[ -V\Big(\frac{x}{\mu}+a\Big) + \alpha  + \beta x\Big]\eta$$
projects symplectically to $0$ (used in the next lemma), in the sense
that 
\begin{equation}
\label{eq:projV}
P \left( i( -V(\bullet/{\mu}+a) + \alpha  + \beta \bullet)\eta \right) = 0 \,, 
\end{equation}
where
$ P  : {\mathcal S}'( \RR, \CC ) \rightarrow {\mathfrak g} $
is defined by the condition that
\[  \omega ( u - P(u)\eta, Y \eta ) = 0\,, \ \  \ \ \forall \, Y \in {\mathfrak g}  \,\]
see \cite[(3.7)]{HZ1} for an explicit expression of $ P $. 
To see \eqref{eq:projV} we use the following simple
\begin{lem}
\label{l:new}
Let $ P $ be the symplectic projection defined above and $ f \in {\mathcal 
S'} ( \RR ; \RR ) $ (that is $ f $ is real valued tempered distribution). 
Then, 
\begin{gather}
\label{eq:ab}
\begin{gathered}
P ( i f ( x) \eta ) =  \alpha e_3 \cdot \eta  + \beta e_2 \cdot   \eta  = 
i \alpha \eta + i \beta x \eta \,,\\
\alpha = \frac 12 \int f ( x ) \eta^2 ( x ) - \frac12\int f'( x) x \eta^2 ( x )  dx
\,, \ \ \beta = \frac12 \int f' ( x ) \eta^2 ( x ) dx \,. 
\end{gathered}
\end{gather}
\end{lem}
\begin{proof}
This follows from a straightforward calculation based on 
\eqref{eq:omij}.
If $ f ( x) = V ( x/\mu + a ) $, we obtain \eqref{eq:projV} with 
$ \alpha $ and $ \beta $ given by \eqref{eq:ab}.
\end{proof}

Finally, note the following Taylor 
expansions:
\begin{align*}
V\Big(\frac{x}{\mu}+a\Big) &= V(a) + V'(a)\frac{x}{\mu} 
+ V''(a)\frac{x^2}{2\mu^2}+\mathcal{O}(h^3) \,, \\
\alpha &= V(a) -\frac{V''(a)\pi^2}{24\mu^2}+\mathcal{O}(h^3) \,, \\
\beta &= \frac{V'(a)}{\mu} + \mathcal{O}(h^3)\,, 
\end{align*}
and thus
\begin{equation}
\label{eq:thus}
-V\Big(\frac{x}{\mu}+a\Big) + \alpha + \beta x 
= -\frac{V''(a)\pi^2}{24\mu^2} - \frac{V''(a)}{2\mu^2}x^2 + \mathcal{O}(h^3)
=\mathcal{O}(h^2)\,,
\end{equation}
where all the errors are polynomially bounded in $ x $. If we assumed
that $ W \in {\mathcal C}^4 $ then he expansions for $ \alpha $ would
be valid with an error $ {\mathcal O} ( h^4 ) $.

\begin{lem}
\label{l:X}
Let $ w $ be given by \eqref{eq:wt}, with $g$ obtained from Lemma \ref{l:FGJS}, and $X$ by \eqref{eq:defX}.  Suppose $\frac12\leq \mu \leq 1$.  
Then
$$|X| \leq c(h^2\|w\|_{H^1} + \|w\|_{H^1}^2 + \|w\|_{H^1}^3)$$
\end{lem}
\begin{proof}
We use the 
symplectic orthogonality of $ Y \eta $, $ Y \in {\mathfrak g} $
and $ w $.

Since $ P w_t = \partial_t P w = 0 $, Lemma \ref{l:wt} gives
\[ \begin{split} X & = 
P \left(i(V(\bullet/{\mu}+a)- \alpha - \beta \bullet ) \eta \right) 
 + P(i( V(\bullet/{\mu}+a ) -\alpha  - \beta \bullet )w ) - P(Xw)
\\ & \ \ \ \ \ +\,  \mu^2P ( i \mathcal{N}w )  + 
\mu^2 P ( i\mathcal{L}w ) \,.
\end{split}
\] 
We recall from \cite[Lemma 3.3]{HZ1} the following straightforward
estimates:
\begin{equation}
\label{eq:estP}
\| P ( Y w ) \| \leq C | Y | \|w \|_{L^2} \,, \ \
 \| P ( i \mathcal{N} u ) \| \leq C \|w\|_{L^2}^2 \left( 1 +
\| w \|_{H^1}^{\frac12} \| w \|_{L^2}^{\frac12} \right) \,. 
\end{equation}

We already observed in \eqref{eq:projV} 
that the first term on the right hand side
vanishes\footnote{This is the significant bonus of using the
effective Hamiltonian.}. From \eqref{eq:thus}
we see that the second term is $ {\mathcal O} ( h^2 ) \| w \|_{L^2} $. 
The third and fourth term are estimated using \eqref{eq:estP} by 
\[ C (  |X| \| w \|_{L^2} + \| w \|_{H^1}^2 + \| w \|_{H^1}^3 ) \,.\]
The last term vanishes: 
the linear operator $ \mathcal{L}$ is the Hessian of $ \mathcal{E} $, given in
\eqref{eq:defE}, at the
critical point $ \eta $. The fact that
 $ \Xi_{\mathcal E} $ is tangent to $ M$ and
Lemma \ref{l:gen2} (or a direct computation) show that
\[  P ( i \mathcal{L} w ) = 0 \,, \]
Summarizing,
\[ |X| \leq C \| w \|_{L^2} | X | + C ( h^2\| w \|_{H^1} +
\| w \|_{H^1}^2 + \| w \|_{H^1}^3 ) \,.\] 
The smallness of $ \| w \|_{L^2} $ concludes the proof.
\end{proof}


We conclude this section which two lemmas which effectively 
eliminate $ \mu $ from the coefficients of $ X$.

\begin{lem}
\label{l:gath}
Suppose that $ w \in H^1 ( \RR, \CC ) $ and that 
$ \omega ( w , X \cdot \eta ) = 0 $,
for every $ X \in {\mathfrak g}$. 
Then 
\begin{equation}
\label{eq:lga1}
\|w\|_{L^2}^2 = 2  (1 - \mu ) / \mu \,. \ \ 
\end{equation}
\end{lem}
\begin{proof}
We first compute
\[ \| \eta + w \|_{L^2}^2 = \| g^{-1} u  \|_{L^2}^2  = 
 \| u  \|_{L^2}^2 / { \mu(g)} = 
{2}/{ \mu(g)} \,,\]
where we used the conservation of the $ L^2 $ norm. By the 
symplectic orthogonality assumption 
statement of the lemma $ \Re \la w , \eta \ra = 0 $ and hence
\[ \| \eta + w \|_{L^2}^2  = 2 + \| w \|_{L^2}^2 \,, \]
from which the conclusion follows. 
\end{proof}

From this we immediately deduce the following
\begin{lem}
\label{C:nomus}
Suppose $1-\mu \ll 1$ and $0 < h \leq 1$.  Let $ X_0 = X |_{\mu = 1} $
where $ X $ is given by \eqref{eq:defX}:
\begin{equation}
\label{eq:defX0}
X_0 \defeq (-\dot a+v) e_1 + \Big(-{\dot v}-\beta\Big)e_2 
+ \Big(-\dot \gamma +\dot a v -\frac12v^2+\frac12-\alpha\Big)e_3
- {\dot \mu} e_4 \,, 
\end{equation}
where $ \alpha $ and $ \beta $ are given by \eqref{eq:albe} (and 
depend on $ \mu $). The the conclusions of Lemma \ref{l:X} hold
for $ X_0$:
$$|X_0 | \leq c(h^2\|w\|_{H^1} + \|w\|_{H^1}^2 + \|w\|_{H^1}^3) \,. $$
\end{lem}

\section{Spectral estimates}
\label{se}


In this section we will recall the now standard estimates on the
operator $ {\mathcal L} $ which arises as Hessian of $ {\mathcal E} $ 
at $ \eta $:
\[ {\mathcal L} w = - \frac12 \partial_x^2 w - 2 \eta^2 w - \eta^2 \bar w 
+ \frac12 w \,, \]
or 
\[ {\mathcal L} w = \begin{bmatrix} L_+ &  0 \\  0 & L_- \end{bmatrix} 
\begin{bmatrix} \Re w \\ \Im w \end{bmatrix} \,, 
 \ \ L_\pm = - \frac 12 \partial_x^2 - ( 2 \pm 1 ) \eta^2 + \frac12 \,.
\]
In our special case we can be more precise than in the general case
(see \cite{We}, and also \cite[Appendix D]{FrSi}). The 
self-adjoint operators
$ L_\pm $ belong to the class of Schr\"odinger operators with 
{\em P\"oschl-Teller} potentials
and their spectra can be explicitly computed:
\[ \sigma( L_- ) = \{ 0 \} \cup [1/2 , \infty ) \,, \ \
\sigma(L_+) = \{ 0, - 3/2\} \cup [1/2, \infty ) \,. \] 
The eigenfuctions can computed by the same method but a straightforward
verification is sufficient to see that 
\[ L_- \eta = 0 \,, \ \ L_+ (\partial_x \eta) = 0 \,, \ \ 
L_+ (\eta^2) = - \frac32 \eta^2 \,.\]

From \S \cite[\S 4]{HZ1} we recall the following
\begin{prop}
\label{p:coer}
Let $ w \in H^1 ( \RR, \CC ) $ and 
suppose that for any $ X \in {\mathfrak g}$,
$ \omega ( w , X \cdot \eta ) = 0 $.
Then, 
\begin{equation}
\label{eq:coer}
\begin{split}
 \langle {\mathcal L} w , w \rangle  \geq 
 \frac{2 \rho_0  } { 7 + 2 \rho_0 }  \| w \|^2_{H^1} 
\simeq 0.0555 \| w \|_{H^1}^2 \,, \ \ \rho_0 = \frac{ 9 }{ 2 ( 12 + \pi^2 )} \,.
\end{split}
\end{equation}
\end{prop}

The next proposition will be useful in solving a linear equation 
for an approximation of $ w $.

\begin{prop}
\label{p:new}
The equation 
\begin{equation}
\label{eq:L+} L_+f= \left( \frac{\pi^2}{12}+x^2 \right)\eta ( x ) 
\end{equation}
has a unique solution in $ L^2 ( \RR ) $. In addition, 
\[  e^{ ( 1 - \epsilon) |x| } f^{(k)} ( x ) \in L^\infty ( \RR ) \,, \ \
\forall \; \epsilon > 0 \,, \ k \in \NN \,, \]
and 
\begin{equation}
\label{eq:orth}
 \omega ( f , X \eta ) = 0 \,, \ \ \forall \; X \in {\mathfrak g}\,. 
\end{equation}
\end{prop}
\begin{proof}
Since $0$ is an 
isolated point of the spectrum of $L_+$ and $L_+$ is self-adjoint, 
$L_+$ has a bounded inverse on 
\[ (\ker L_+)^\perp = (\operatorname{span}\{\partial_x \eta\})^\perp \,.\]
Hence, 
\[    (\ker L_+)^\perp  \ni   \left(\frac{\pi^2}{12}+x^2 \right)\eta  
\; \stackrel{L_+^{-1}}{\longmapsto} \;  f \in  (\ker L_+)^\perp  \,. \] 

By elliptic 
regularity, $f$ is smooth.  Moreover, since $ L_+ $ commutes
with $ g ( x ) \mapsto g ( -x ) $ and 
$ (\frac{\pi^2}{12}+x^2)\eta ( x ) $ is even, 
we conclude that $f$ is even.

Next we argue that 
$f$ has exponential decay. 
For that we conjugate the equation by 
$e^{\sigma x}$:
\[ e^{  \sigma x } L_+ e^{ - \sigma x } = L_+ +   \sigma \partial_x - 
\sigma^2 / 2 \,, \]
and apply both sides to $ e^{\sigma x } f ( x)$:
\[ e^{\sigma x } \Big( \frac{\pi^2}{12} + x^2 \Big) \eta 
= \Big(-\partial_x^2 + 2 \sigma \partial_x -  
\frac{\sigma^2}2 + \frac12 \Big)  e^{\sigma x } f + 3 \eta^2 e^{\sigma x } f 
\,.\]
Taking the Fourier transform this means that we have
$$\left( \frac 12 \xi^2  +\frac12-\frac12\sigma^2+i\sigma \xi \right ) 
\widehat{e^{\sigma x}f}(\xi) 
= \Big[ 3e^{\sigma x}\eta^2f 
+ e^{\sigma x}\Big(\frac{\pi^2}{12}+x^2\Big)\eta\Big]^{\widehat{\;}}(\xi),$$
which makes sense if $ |\sigma | < 1 $ as then the right hand side is
in $ L^2 $. 
Since the multiplier on the left-hand side is bounded away from $0$, we can 
invert the expression to obtain that $e^{\sigma x}f  \in H^2$.  From this, we 
deduce by Sobolev embedding that $e^{\sigma x}f\in L^\infty$.  By applying 
derivatives to \eqref{eq:L+} and then 
repeating this argument, we in fact obtain that all derivatives are pointwise 
exponentially localized in space.  

Finally, we prove the orthogonality condition \eqref{eq:orth}, that 
is,  $ \omega ( f , e_j \cdot \eta ) = 0 $ where
$ e_j $'s are given by \eqref{eq:liea1}.
Since $f$ is real, we clearly have 
\[ \omega(f, \eta') = \omega(f, (x\eta)')=0\,. \]
Since $f$ is even, we also have $\omega(f, ix\eta) =0$.  
It remains to show that $\omega(f,i\eta)=0$, that is that 
$\int f\eta=0$.  Note that 
$L_+ [(x \eta)']=\eta$ (by direct computation or from the structure of the
generalized kernel of $ i {\mathcal L} $).    Hence
$$\int f\eta = \int f L_+[(x\eta)'] \,dx = \int L_+f \; (x\eta)' \,dx 
= \int \Big(\frac{\pi^2}{12}+x^2\Big)\eta (\eta+ x\eta')\, dx\,. $$
Integration by parts and $\int \eta^2=2$, 
$\int x^2\eta^2 = {\pi^2}/{6}$, show that this is $0$.
\end{proof}


\renewcommand\thefootnote{\ddag}%

\section{Proof  of the main estimate}
\label{pr}

In the arguments that follow, we will assume that $g(u)$ satisfying \eqref{E:gdef} is always defined on any time interval under consideration, and thus $w$ given by \eqref{eq:wt} is also always defined.  This is, however, not known \textit{a priori} and Lemma \ref{l:FGJS} must be considered 
as part of the bootstrap argument, together with the lemmas that follow.  However, since this aspect of the argument is standard in papers on this subject, we will not make mention of it.  

Recall from Lemma \ref{l:wt} that the equation for $w$ is
\begin{equation}
\label{E:w}
\partial_t w =
\begin{aligned}[t]
&X\eta +i\Big[-V\Big(\frac{x}{\mu}+a\Big)+\alpha +\beta x\Big]\eta \\
&+Xw+i\Big[-V\Big(\frac{x}{\mu}+a\Big)+\alpha +\beta x\Big]w 
-i\mu^2\mathcal{L}w +i\mu^2\mathcal{N}w\,,
\end{aligned}
\end{equation}
where $ \alpha $ and $ \beta $ are time dependent parameters given by 
\eqref{eq:albe}.
Note that by Taylor's theorem the forcing term in this equation has second-order expansion
$$-i\Big[ V\Big(\frac{x}{\mu}+a\Big)+ \alpha +\beta x\Big]\eta 
= -i\frac{V''(a)}{2\mu^2}\Big(\frac{\pi^2}{12}+x^2\Big)\eta 
+ h^3\la x\ra^3\eta f_2$$
where, provided\footnote{This will follow easily from the bootstrap assumption.} $ 1/2\leq \mu\leq 1$, we have $|f_2(x,t)|\leq c$.  

We first 
consider the forced linear evolution obtained from \eqref{E:w} by 
discarding all terms we expect will be of order $h^3$ or higher:
\begin{equation}
\label{E:forcedlinear}
\partial_t \tilde w =-i\mu^2\mathcal{L}\tilde w 
-i\frac{V''(a)}{2\mu^2}\Big(\frac{\pi^2}{12}+x^2\Big)\eta
\end{equation}
We now introduce a natural approximate solution to this forced linear equation.  Let 
\begin{equation}
\label{eq:defwt}
\tilde w = -\frac{V''(a)}{2\mu^4}f\,, \ \ 
f \defeq L_{+}^{-1} \left(\frac{\pi^2}{12}+x^2 \right)\eta \,, 
\end{equation}
where $ f $ is given in Proposition \ref{p:new}.
Then $\tilde w$ satisfies \eqref{E:forcedlinear} to second-order, i.e.
\begin{equation}
\label{E:tildew}
\partial_t \tilde w = -i\mu^2\mathcal{L}\tilde w 
-i\frac{V''(a)}{2\mu^2}\Big(\frac{\pi^2}{12}+x^2\Big)\eta + h^3\theta f
\end{equation}
where 
$$\theta(t) \defeq \frac{1}{h^3}\Big[ -\frac{V'''(a)\dot a}{2\mu^4} 
+ \frac{2V''(a)\dot \mu}{\mu^5} \Big]$$

\begin{lem}
\label{L:thetabound}
Suppose there is a constant $c_1$ such that
\begin{equation}
\label{E:bs}
\|w\|_{L_{[t_1,t_2]}^\infty H_x^1} \leq c_1h^{3/2}
\end{equation}
Then provided $|t_2-t_1| \leq h^{-2}$, we have
$$\sup_{t_1\leq t\leq t_2}|\theta(t)| \leq c \, ,$$
where $c$ is a constant depending only upon $c_1$, $\|W^{(k)}\|_{L^\infty}$ for $k=0,1,2,3$ and $|v(t_1)|$. 
\end{lem}
\begin{proof}
By \eqref{E:bs}, Lemma \ref{C:nomus} and Taylor expansions 
(see \eqref{eq:thus}), we have
\begin{equation}
\label{E:paramcontrol}
|\dot v + V'(a)| +|\dot \mu|  + |-\dot a+v| \leq ch^3
\end{equation}
By integrating the first bound, we obtain the following rough estimate
\begin{equation}
\label{E:vcrude}
\sup_{t_1\leq t\leq t_2} |v(t)| \leq |v(t_1)| + ch\|W'\|_{L^\infty}(t_2-t_1) + ch^3(t_2-t_1) \,. 
\end{equation}
We have
$$|\dot v v + V'(a)\dot a| \leq  |\dot v + V'(a)||v| + |V'(a)||\dot a -v|$$
and thus, by \eqref{E:paramcontrol} and \eqref{E:vcrude},
$$\sup_{t_1\leq t\leq t_2} |\dot v v + V'(a)\dot a| \leq ch^3|v(t_1)|+ ch^4\|W'\|_{L^\infty}\la t_2-t_1\ra + ch^6(t_2-t_1)$$
Integrating this bound, we obtain a near conservation of classical energy,
\begin{align*}
\indentalign \Big| \Big( \frac{v^2}{2} + V(a)\Big) - \Big( \frac{v(t_1)^2}{2} + V(a(t_1))\Big) \Big| \\
&\leq ch^3|v(t_1)|(t_2-t_1) + ch^4\|W'\|_{L^\infty}\la t_2-t_1\ra(t_2-t_1) + ch^6(t_2-t_1)^2
\end{align*}
from which we obtain $\sup_{t_1\leq t\leq t_2} |v(t)| \leq c$.  
This and \eqref{E:paramcontrol} imply that 
$$\sup_{t_1\leq t\leq t_2} |\dot a(t)| \leq c\,, $$
 and thus $|\theta(t)|\leq c$.
\end{proof}

This approximate solution $\tilde w$ provides heuristic evidence that $w$ should be of order $h^2$, but it will also play a key r\^ole in our rigorous argument establishing this fact. 

\begin{lem}[Lyapunov energy estimate]
\label{L:Lyapunov}
Suppose that, for some constant $c_1$,
\begin{equation}
\label{E:bootstrap}
\|w\|_{L_{[t_1,t_2]}^\infty H_x^1} \leq c_1h^{3/2}
\end{equation}
Then, provided 
\begin{equation}
\label{E:timebd}
|t_2-t_1| \leq \frac{c}{h} \, ,
\end{equation}
we have,
\begin{equation}
\label{E:w_est}
\|w\|_{L^\infty_{[t_1,t_2]} H^1} \leq 4\sqrt{c_2}\|w(t_1)\|_{H^1} +ch^2
\end{equation} 
The constants $c$ in \eqref{E:timebd} and \eqref{E:w_est} depend upon $c_1$, $\|W^{(k)}\|_{L^\infty}$ for $k=0,1,2,3$, and $|v(t_1)|$, although are independent of $\|w(t_1)\|_{H^1}$.  The constant $c_2 \defeq (7+2\rho_0)/(2\rho_0)\approx 18.02$, with $\rho_0$ the absolute constant appearing in \eqref{eq:coer}. 
\end{lem}

We will postpone the proof of Lemma \ref{L:Lyapunov} to the end of the section.
In the next theorem we iterate the above bound, and close the bootstrap
argument.

\begin{thm}
\label{t:2}
Let $0\leq \delta < \frac14$.  Let $w_0=w(0)$, and suppose that there is a constant $c_1$ such that 
$$\|w_0\|_{H^1} \leq c_1h^{\frac32+3\delta} \,.$$   
Then, provided 
\begin{equation}
\label{E:tbound}
|t|\leq \frac{c(1+\delta|\log h|)}{h} \, \; \text{ and } \; 0<h\leq \epsilon \, ,
\end{equation}
we have
\begin{equation}
\label{E:wbd2}
\|w\|_{L_{[0,t]}^\infty H_x^1} \leq 4\sqrt{c_2}h^{-2\delta}\|w_0\|_{H^1} + ch^{2(1-\delta)}
\end{equation}
The constant $c$ in \eqref{E:tbound} and \eqref{E:wbd2} and the constant $\epsilon$ in \eqref{E:tbound} depend upon $c_1$, $\|W^{(k)}\|_{L^\infty}$ for $k=0,1,2,3$, and $|v(0)|$, although are independent of $\delta$ and $\|w_0\|_{H^1}$.  The constant $c_2 \defeq (7+2\rho_0)/(2\rho_0)\approx 18.02$, with $\rho_0$ the absolute constant appearing in \eqref{eq:coer}. 
\end{thm}
\begin{proof}
We apply Lemma \ref{L:Lyapunov} $k$ times on successive intervals each of size $c/h$ (where $c$ is as given in Lemma \ref{L:Lyapunov}) to obtain the bound
$$\|w\|_{L_{[0,ck/h]}^\infty H_x^1} \leq (4\sqrt c_2)^k\|w_0\|_{H^1} + \left( \sum_{j=0}^{k-1} (4\sqrt c_2)^j \right)ch^2$$
This is only valid provided that the hypothesis of Lemma \ref{L:Lyapunov}
is satisfied  over the whole collection of time intervals:
$$ (4\sqrt c_2)^k\|w_0\|_{H^1} + \left( \sum_{j=0}^{k-1} (4\sqrt c_2)^j \right)ch^2 \leq c_1h^{3/2} \,,$$
By taking 
$$k= 1+\frac{2 \delta|\log h|}{\log (4\sqrt {c_2})} \,$$
we obtain that
$$(4\sqrt c_2)^k\|w_0\|_{H^1} + \left( \sum_{j=0}^{k-1} (4\sqrt c_2)^j \right)ch^2 \leq (4\sqrt c_2)h^{-2\delta}\|w_0\|_{H^1} + ch^{2(1-\delta)}$$
Thus, it suffices to ensure that
$$(4\sqrt c_2)h^{-2\delta}\|w_0\|_{H^1} + ch^{2(1-\delta)} \leq c_1h^{3/2}\,, $$
and 
this is accomplished provided $h\leq \epsilon$, for a suitable $\epsilon$ with the dependence as stated in the proposition.  We note that the constant $c$ from Lemma \ref{L:Lyapunov} did not change from one iteration to the next, since $|v(t_1)|$ remains uniformly bounded by Lemma \ref{L:thetabound}, which applies on a time interval of size $h^{-2}$.
\end{proof}

We now conclude the section with the proof of Lemma \ref{L:Lyapunov}.

\begin{proof}[Proof of Lemma \ref{L:Lyapunov}]
For the proof, we shall assume, in place of \eqref{E:bootstrap}, the bound
$$\|w\|_{L_{[t_1,t_2]}^\infty H^1} \leq c_1h^{2(1-\delta')}$$
We will conclude at the end that it suffices to take $\delta'=\frac14$.  Let 
$$w_1\defeq w-\tilde w \, ,$$
where $ \tilde w $ is an approximate solution to \eqref{E:forcedlinear}
given by \eqref{eq:defwt}.

From \eqref{E:w} and \eqref{E:tildew}, we derive the equation 
satisfied by $ w_1 $:
$$\partial_t w_1 =
\begin{aligned}[t]
&-i\mu^2\mathcal{L}w_1+ X\eta + i\Big[- V\Big(\frac{x}{\mu}+a\Big)+\alpha 
+\beta x +\frac{V''(a)}{2\mu^2}\Big(\frac{\pi^2}{12}+x^2\Big)\Big]\eta 
- h^3\theta f\\
&+Xw+i\Big[-V\Big(\frac{x}{\mu}+a\Big)+\alpha +\beta x\Big]w  
+i\mu^2\mathcal{N}w \,.
\end{aligned}
$$
By grouping forcing terms of size $ {\mathcal O} (h^3) $, 
we rewrite the above as
$$\partial_t w_1 = -i\mu^2\mathcal{L}w_1+ h^3f_1+Xw
+i\Big[-V\Big(\frac{x}{\mu}+a\Big)+\alpha +\beta x\Big]w  +i\mu^2\mathcal{N}w
$$
where $\|f_1(x,t)\|_{H_x^1} \leq c$ (uniformly in $t$).  Note that here we have applied Lemma \ref{L:thetabound} to conclude that $|\theta(t)|\leq c$.  

We recall that $\mathcal{L}$ is self-adjoint 
with respect to
\[\la u,v\ra = \Re \int u\bar v \,, \]
and hence, writing $Xw=Xw_1+X\tilde w$, we compute 
\begin{align*}
\frac 12 \partial_t \la \mathcal{L}w_1, w_1\ra &= \la \mathcal{L}w_1,\partial_t w_1\ra\\
&= 
\begin{aligned}[t]
&-\mu^2\la \mathcal{L}w_1, i\mathcal{L}w_1\ra + \la \mathcal{L}w_1,h^3f_1\ra 
+ \la \mathcal{L}w_1, Xw_1\ra + \la \mathcal{L}w_1, X\tilde w\ra \\
&+  \la \mathcal{L}w_1, 
i\Big[-V\Big(\frac{x}{\mu}+a\Big)+\alpha +\beta x\Big]w_1 \ra \\
&+ \la \mathcal{L}w_1, 
i\Big[-V\Big(\frac{x}{\mu}+a\Big)+\alpha +\beta x\Big]\tilde w \ra 
+ \mu^2\la \mathcal{L}w_1, i\mathcal{N}w \ra
\end{aligned}\\
&= \text{I}+\text{II}+\text{III}+\text{IV}+\text{V}+\text{VI}+\text{VII}
\end{align*}
Now we analyse these terms one-by-one.  
First, 
\begin{equation}
\label{E:bdI}
\text{I}=0 . 
\end{equation}
For II, we use integration by parts for the $\partial_x^2$ term to move one 
$\partial_x$ onto $f_1$ and then apply the Cauchy-Schwarz
inequality; for the other terms we use a 
direct application of the Cauchy-Schwarz inequality, and together
these give
\begin{equation}
\label{E:bdII}
|\text{II}| \leq 4h^3 \|w_1\|_{H^1}\|f_1\|_{H_x^1} \leq ch^3\|w_1\|_{H^1}
\end{equation}
The next term, $\text{III}$, requires more care:
\begin{align*}
\text{III} &= \la \mathcal{L}w_1, Xw_1 \ra \\
&= \frac 12 \la (  w_1 -  \partial_x^2w_1- 4 \eta^2w_1- 2 \eta^2\bar w_1),
(-a_1\partial_xw_1+ a_2ixw_1 + a_3iw_1 +a_4\partial_x (xw_1))\ra
\end{align*}
where $a_j$, the components of $X$, are time dependent but space independent.  
By the bootstrap assumption \eqref{E:bootstrap} and Lemma \ref{l:X}, $|a_j|\leq ch^{4-4\delta'}$.  To proceed, we do some  further calculations:
$$\la w_1, Xw_1\ra = a_4 \la w_1, (w_1 + x\partial_x w_1)\ra 
= \frac12a_4\|w_1\|_{L^2}^2$$
\begin{align*}
\la \partial_x^2 w_1, Xw_1 \ra &= a_2\la \partial_x^2 w_1, ixw_1\ra 
+ a_4 \la \partial_x^2 w_1, w_1+x\partial_x w_1\ra \\
&=a_2\la \partial_x w_1, i w_1\ra - \frac32a_4 \|\partial_x w_1\|_{L^2}^2
\end{align*}
and thus the above two terms are bounded by $ch^{4-4\delta'}\|w_1\|_{H^1}^2$.  
For the terms involving $ \eta $, we use the Cauchy-Schwarz inequality 
and the fact that  $\eta \in {\mathcal S}$:
$$|\la \eta^2 w_1, Xw_1\ra| 
+ |\la \eta^2 w_1, X\bar w_1\ra|\leq (\max |a_j|) \|w_1\|_{H^1}^2$$
Altogether then, we have
\begin{equation}
\label{E:bdIII}
|\text{III}| \leq ch^{4-4\delta'}\|w_1\|_{H^1}^2
\end{equation}
Now we move on to IV:
\begin{align*}
\text{IV} &= \la \mathcal{L}w_1, X\tilde w\ra \\
&= \frac 12 \la (w_1-\partial_x^2w_1-4 \eta^2w_1- 2 \eta^2\bar w_1),
(-a_1\partial_x\tilde w+ a_2ix\tilde w + a_3i\tilde w 
+a_4\partial_x (x\tilde w))\ra \,. 
\end{align*}
For this term, we are forced to directly estimate
by the Cauchy-Schwarz inequality 
(although for the $\partial_x^2w_1$ term, we integrate by 
parts one $\partial_x$ factor).   The bound that we get is
$$|\text{IV}| \leq ch^{4-4\delta'}\|w_1\|_{H^1} (\|\la x \ra \tilde w\|_{L^2} 
+ \|\la x \ra \partial_x \tilde w\|_{L^2} 
+ \|\la x \ra \partial_x^2 \tilde w\|_{L^2})$$
From the definition \eqref{eq:defwt} of 
$\tilde w$ and Proposition \ref{p:new}, 
we obtain that
all the norms involving $\tilde w$ are bounded by $h^2$.  Thus,
\begin{equation}
\label{E:bdIV}
|\text{IV}| \leq ch^{6-4\delta'}\|w_1\|_{H^1} \,. 
\end{equation}
Next, we move on to V:
\begin{align*}
\text{V} &= \la \mathcal{L}w_1, i\Big[-V\Big(\frac{x}{\mu}+a\Big)+\alpha 
+\beta x\Big]w_1 \ra \\
&= \frac 12 \la (w_1-\partial_x^2w_1- 4\eta^2w_1-2 \eta^2\bar w_1),  
i\Big[-V\Big(\frac{x}{\mu}+a\Big)+\alpha +\beta x\Big]w_1 \ra \\
&= \frac 12 
\la (-\partial_x^2w_1-\eta^2\bar w_1),  i\Big[-V\Big(\frac{x}{\mu}+a\Big)
+\alpha +\beta x\Big]w_1 \ra \,. 
\end{align*}
To estimate the first term, we integrate by parts and use that 
$$\Big| -\frac{1}{\mu}V'\Big(\frac{x}{\mu}+a\Big) + \beta \Big| \leq ch$$
(note that other estimates are available, like $cxh^2$, but we do not want an 
$x$ coefficient here).  For the second term, we use \eqref{eq:thus},
the Cauchy-Schwarz 
inequality and the rapid decay of 
$\eta^2$:
$$\Big| \Big[-V\Big(\frac{x}{\mu}+a\Big)+\alpha +\beta x\Big]\eta^2 \Big| 
\leq ch^2$$
This gives the bound\footnote{It is unlikely that we can do better than $h$ as a coefficient here, and thus 
this seems to be what limits us ultimately to time $1/h$.}  
\begin{equation}
\label{E:bdV}
|\text{V}| \leq ch\|w_1\|_{H^1}^2 \,. 
\end{equation}
 Now we move on to 
the next term, VI.  
$$\text{VI} = \la \mathcal{L}w_1, i \Big[ -V\Big(\frac{x}{\mu}+a\Big)+\alpha
+\beta x\Big]\tilde w \ra$$
In the $\partial_x^2$ term of $\mathcal{L}$, we integrate by parts one 
$\partial_x$, and then estimate by the Cauchy-Schwarz inequality. 
All other terms, are estimated by a direct application of 
the Cauchy-Schwarz inequality. The bound obtained is
\begin{align}
\notag |\text{VI}| &\leq \|w_1\|_{H^1}\left\| \la \partial_x \ra 
\Big[ -V\Big(\frac{x}{\mu}+a\Big)+\alpha+\beta x\Big]\tilde w\right \|_{L^2} \\
\notag & \leq ch^2\|w_1\|_{H^1}(\| \la x\ra^2\tilde w\|_{L^2} 
+ \|\la x\ra^2\partial_x\tilde w\|_{L^2})\\
\label{E:bdVI} &\leq ch^4\|w_1\|_{H^1}
\end{align}
by the localization of $\tilde w$.  For the last term, VII, we use integration 
by parts once for the $\partial_x^2$ term, and then apply 
the Cauchy-Schwarz inequality to all 
terms.  Since we are in one-dimension, we have 
the embedding $\|w\|_{L^\infty} \leq c\|w\|_{H^1}$.
$$\text{VII} = -\mu^2\la \mathcal{L}w_1, i\mathcal{N}w\ra$$
\begin{equation}
\label{E:bdVII}
\implies |\text{VII}| \leq \|w_1\|_{H^1}( \|w\|_{H^1}^2+ \|w\|_{H^1}^3) 
\leq ch^{4-4\delta'}\|w_1\|_{H^1}
\end{equation}
by the bootstrap assumption. 

This completes the step-by-step estimation process, and the bound we get from 
\eqref{E:bdI},  \eqref{E:bdII},  \eqref{E:bdIII},  \eqref{E:bdIV},  
\eqref{E:bdV},  \eqref{E:bdVI},  \eqref{E:bdVII} 
is
$$|\partial_t \la \mathcal{L}w_1, w_1 \ra| \leq c(h^{4-4\delta'}+h^3)\|w_1\|_{H^1} + c(h^{4-4\delta'}+h)\|w_1\|_{H^1}^2$$
We see that it suffices to take $\delta'=\frac14$.  Integrating in time, we get
\begin{equation}
\label{eq:inte}
\la \mathcal{L}w_1(t),w_1(t)\ra \leq 
\begin{aligned}[t]
&\la \mathcal{L}w_1(t_1),w_1(t_1)\ra + c(t-t_1)h^3\|w_1\|_{L_{[t_1,t_2]}^\infty H^1} \\
&+ c(t-t_1)h\|w_1\|_{L_{[t_1,t_2]}^\infty H^1}^2
\end{aligned}
\end{equation}
By \eqref{E:gdef} and \eqref{eq:orth}, $w_1(t)$ satisfies the hypothesis of 
Proposition \ref{p:coer}, and we have
\begin{equation}
\label{E:specbd}
\frac{1}{c_2}\|w_1(t)\|_{H^1}^2 \leq \la \mathcal{L}w_1(t),w_1(t) \ra
\end{equation}
By direct estimation, we have the upper bound
$$|\la \mathcal{L}w_1(t_1),w_1(t_1)\ra| \leq 4\|w_1(t_1)\|_{H^1}^2$$
Combining this with \eqref{eq:inte} we get the bound
$$\|w_1(t)\|_{H^1}^2 \leq 
\begin{aligned}[t]
&4c_2\|w_1(t_1)\|_{H^1}^2 
+ c(t-t_1)h^3\|w_1\|_{L_{[t_1,t_2]}^\infty H^1} \\
&+ c(t-t_1)h\|w_1\|_{L_{[t_1,t_2]}^\infty H^1}^2
\end{aligned}$$
From this, we infer from the monotonicity of the right side that
$$\|w_1\|_{L_{[t_1,t_2]}^\infty H^1}^2 \leq 
\begin{aligned}[t]
&4c_2\|w_1(t_1)\|_{H^1}^2 
+ c(t_2-t_1)h^3\|w_1\|_{L_{[t_1,t_2]}^\infty H^1} \\
&+ c(t_2-t_1)h\|w_1\|_{L_{[t_1,t_2]}^\infty H^1}^2
\end{aligned}$$
Requiring that $t_2-t_1\leq {c}/{h}$ implies
\[ \begin{split} \|w_1\|_{L_{[t_1,t_2]}^\infty H^1}^2 
& \leq 8c_2\|w_1(t_1)\|_{H^1}^2+ch^2\|w_1\|_{L_{[t_1,t_2]}^\infty H^1}\\
& \leq 16c_2\|w_1(t_1)\|_{H^1}^2+ch^4 \,. \end{split} \]
Since $w=w_1+\tilde w$ and $\|\tilde w\|_{H^1} \leq ch^2$,
\begin{equation}
\label{E:w_estimate}
\|w\|_{L^\infty_{[t_1,t_2]} H^1} \leq 4\sqrt{c_2}\|w(t_1)\|_{H^1} + ch^2 \,, 
\end{equation}
which is the claimed estimate.
\end{proof}

\section{ODE analysis}
\label{ode}

To pass from an approximate equations for the parameters 
of the soliton, $ ( a, v, \gamma, \mu ) $ given in Lemma \ref{C:nomus}, 
to the ODEs \eqref{eq:t3} we need some 
elementary estimates which we present in this section. They are similar to 
those in \cite[\S 7]{HZ1}.

\begin{lem}
\label{L:ODEcompare}
Suppose that $0 < h \ll 1$, and $a=a(t)$, $v=v(t)$, 
$\epsilon_1=\epsilon_1(t)$, $\epsilon_2=\epsilon_2(t)$ are $C^1$ real-valued 
functions.  Suppose $f:\mathbb{R}\to\mathbb{R}$ is a $C^2$ mapping such that 
$|f|$ and $|f'|$ are uniformly bounded.  Suppose that on $[0,T]$, 
\begin{equation}
\label{E:ODE}
\left\{
\begin{aligned}
&\dot a = v + \epsilon_1\\
&\dot v = hf(ha) + \epsilon_2
\end{aligned}
\right., \qquad
\begin{aligned}
&a(0)=a_0\\
&v(0)=v_0
\end{aligned}
\end{equation}
Let $\bar a=\bar a(t)$ and $\bar v=\bar v(t)$ be the $C^1$ real-valued 
functions satisfying the exact equations
$$
\left\{
\begin{aligned}
&\dot{\bar a} = \bar v \\
&\dot{\bar v} = hf(h\bar a) 
\end{aligned}
\right., \qquad
\begin{aligned}
&\bar a(0)=a_0\\
&\bar v(0)=v_0
\end{aligned}
$$
with the same initial data.  Suppose that on $[0,T]$, we have 
$|\epsilon_j| \leq h^{4-\delta}$ for $j=1,2$.  Then provided 
$T\leq \delta h^{-1}\log (1/h)$, we have on $[0,T]$ the estimates
$$|a-\bar a|\leq h^{2-2\delta}\log(1/ h), 
\qquad |v-\bar v| \leq h^{3-2\delta}\log(1/h)$$
\end{lem}

Before proceeding to the proof, we recall some basic tools.

\noindent\textit{Gronwall estimate}.  Suppose $b=b(t)$ and $w=w(t)$ are $C^1$ 
real-valued functions, $h$ is a constant, and $(b,w)$ satisfy the differential 
inequality:
\begin{equation}
\label{E:diff_ineq}
\left\{
\begin{aligned}
&|\dot b| \leq |w| \\
&|\dot w| \leq h^2 |b|
\end{aligned}
\right. ,
\qquad
\begin{aligned}
&b(0)=b_0\\
&w(0)=w_0
\end{aligned}
\end{equation}
Let $x(t)= hb(t/h)$, $y(t)=w(t/h)$.  Then 
$$
\left\{
\begin{aligned}
&|\dot x| \leq |y| \\
&|\dot y| \leq |x|
\end{aligned}
\right. ,
\qquad
\begin{aligned}
&x(0)=x_0= h b_0\\
&y(0)=y_0=w_0
\end{aligned}
$$
Let $z(t)=x^2+y^2$.  Then $|\dot z| = |2x\dot x + 2y\dot y| \leq 2|x||y| 
+ 2|x||y| \leq 2(x^2+y^2) = 2z$, and hence $z(t) \leq z(0)e^{2t}$.
Thus
$$
\begin{aligned}
&|x(t)| \leq \sqrt 2\max(|x_0|,|y_0|) \exp(t)\\
&|y(t)| \leq \sqrt 2\max(|x_0|,|y_0|) \exp(t)
\end{aligned}
$$
Converting from $(x,y)$ back to $(b,w)$, we obtain the Gronwall estimate
\begin{equation}
\label{E:Gron}
\begin{aligned}
&|b(t)| \leq \sqrt 2\max(h|b_0|,|w_0|)
\frac{\exp(ht)}{h}\\
&|w(t)| \leq \sqrt 2\max(h|b_0|,|w_0|)\exp(ht)
\end{aligned}
\end{equation}

\noindent\textit{Duhamel's formula}.
For a two-vector function $X(t): \mathbb{R} \to \mathbb{R}^2$, a two-vector 
$X_0\in \mathbb{R}^2$, and a $2\times 2$ matrix function 
$A(t):\mathbb{R}\to (2\times 2\text{ matrices})$, let $X(t)=S(t,t')X_0$ denote 
the solution to the ODE system $\dot X(t) = A(t)X(t)$ with $X(t')=X_0$:
\[ \frac{d}{dt} S(t,t')X_0 = A(t)S(t,t')X_0 \,, \ \ S(t',t')X_0=X_0 \, . \]
Then, for a given two-vector function $f(t):\mathbb{R}\to \mathbb{R}^2$, the 
solution to the inhomogeneous ODE system  
\begin{equation}
\label{E:inhomODE}
\dot X(t) = A(t)X(t) + F(t)
\end{equation}
with initial condition $X(0)=0$ is given by Duhamel's formula
\begin{equation}
\label{E:Duhamel}
X(t) = \int_0^t S(t,t')F(t')dt'
\end{equation}

\begin{proof}[Proof of Lemma \ref{L:ODEcompare}]
Let $\tilde a= a-\bar a$ and $\tilde v = v-\bar v$; these perturbative 
functions satisfy
$$
\left\{
\begin{aligned}
&\dot{\tilde a} = \tilde v + \epsilon_1\\
&\dot{\tilde v} =  h^2 g \tilde a + \epsilon_2
\end{aligned}
\right., \qquad
\begin{aligned}
&\tilde a(0)=0\\
&\tilde v(0)=0
\end{aligned}
$$
where $g=g(t)$ is given by
$$g=\left\{
\begin{aligned}
& \frac{f( ha) -f(h \bar a)}{h(a-\bar a)} & \ \ \text{if }\bar a \neq a\\
& \ \ f'(ha) &\ \ \text{if }a=\bar a
\end{aligned}
\right.
$$
which is $C^1$ (in particular, uniformly bounded).  Set 
$$A(t) = \begin{bmatrix} 0 & 1 \\ h^2g(t) & 0 \end{bmatrix}, 
\quad F(t) = \begin{bmatrix} \epsilon_1(t) \\ \epsilon_2(t) \end{bmatrix}, 
\quad X(t)=\begin{bmatrix} \tilde a(t) \\ \tilde v(t) \end{bmatrix}$$
in \eqref{E:inhomODE}, and appeal to Duhamel's formula \eqref{E:Duhamel} to 
obtain
\begin{equation}
\label{E:Duhamel2}
\begin{bmatrix}
\tilde a(t) \\ \tilde v(t) 
\end{bmatrix}
= \int_0^t S(t,t') \begin{bmatrix} \epsilon_1(t') \\ \epsilon_2(t') 
\end{bmatrix} \, dt'
\end{equation}
Apply the Gronwall estimate \eqref{E:Gron} with
$$\begin{bmatrix} b(t) \\ w(t) \end{bmatrix} = S(t+t',t')\begin{bmatrix} 
\epsilon_1(t') \\ \epsilon_2(t') \end{bmatrix},  \quad \begin{bmatrix} b_0 
\\ w_0 \end{bmatrix} = \begin{bmatrix} \epsilon_1(t') \\ 
\epsilon_2(t') \end{bmatrix} $$
to conclude that 
$$\left| S(t,t') \begin{bmatrix} \epsilon_1(t') \\ 
\epsilon_2(t') \end{bmatrix} \right| \leq \sqrt 2 \begin{bmatrix} 
h^{-1}\exp(h(t-t')) \\ \exp(h(t-t')) \end{bmatrix} 
\max(h|\epsilon_1(t')|,|\epsilon_2(t')|)$$
Feed this into \eqref{E:Duhamel2} to obtain that on $[0,T]$
\begin{align*}
&|\tilde a(t)| \leq \sqrt 2 \, T\frac{\exp(hT)}{h} 
\sup_{0\leq s\leq T}\max(h|\epsilon_1(s)|,|\epsilon_2(s)|)\\
&|\tilde v(t)| \leq \sqrt 2 \, T\exp( h T) 
\sup_{0\leq s\leq T}\max(h |\epsilon_1(s)|,|\epsilon_2(s)|)
\end{align*}
Taking $T\leq \delta h^{-1}\log(1/h)$, we obtain the claimed bounds.
\end{proof}

\section{Proof of Theorem \ref{t:1}}
\label{prt1}

We can now put all the components of the proof together.
Lemma \ref{C:nomus} and Theorem \ref{t:2} show that 
on the time interval $ 0 < t < c \delta \log ( 1/h ) / h $
we have \eqref{eq:t1} with the parameters satisfying
\begin{gather*}
\dot {  a} =   v + {\mathcal O} (h^{4( 1 - \delta) } )\,, \ \ 
\dot{  v} = - \sech^2 * V'  (  a) /2 + {\mathcal O} (h^{4( 1 - \delta) } )
\,,  \ \ \dot \mu = {\mathcal O} (h^{4( 1 - \delta) } ) \,, 
\\ 
\dot {  \gamma} = 
1/2  + 
{v^2}/2 
-  \sech^2 * V (   a ) +  ( x \, \sech^2 x \tanh x ) * V  (  a) 
+ {\mathcal O} (h^{4( 1 - \delta) } ) \,.
\end{gather*}
Lemma \ref{L:ODEcompare} can be applied to replace 
$ a $ and $ v $ with solutions of \eqref{eq:t3} and the direct 
integration of the error terms shows that the same is true for 
$ \mu $ and $ \gamma $. In particular we can drop $ \mu $ altogether. 
\stopthm

We conclude the paper with some remarks. The proof above and Theorem
\ref{t:2} show that the conclusions of Theorem \ref{t:1} remain 
unchanged if instead of taking $ e^{i x v_0} \sech ( x - x_0 ) $ as
initial condition, we took 
\[ e^{i x v_0} \sech ( x - x_0 ) + r ( x ) \,, \ \ \ \| r \|_{H^1} \leq C 
h^{ 2 - \delta } \,.\]
We could go down to $ \| r \|_{H^1} \leq C h^{ 3/2 + 3 \delta } $ at
the expense of complicating the final statement to \eqref{E:wbd2}.
A more general condition on the initial value would make the 
bootstrap argument in \S \ref{pr} so unwieldy that we opted out 
of pursuing that technical issue. 

In higher dimensions similar methods are clearly applicable for 
weaker nonlinearities and under further spectral assumptions -- 
see \cite{FrSi} for examples. At this early stage we restrict
ourselves to the physically relevant cubic nonlinearity which 
at the moment is tracktable only in dimension one.

\end{document}